\documentclass[preprint]{elsarticle}
\usepackage{amsfonts,amsmath,amssymb,amscd,amsthm}
\usepackage[all]{xy}

\newtheorem{theorem}{Theorem}
\newtheorem{prop}[theorem]{Proposition}
\newtheorem{lem}[theorem]{Lemma}

\newdefinition{rem}[theorem]{Remark}
\newdefinition{mydef}[theorem]{Definition}
\newdefinition{example}[theorem]{Example}
\newdefinition{examples}[theorem]{Examples}

\def\C{\mathbb{C}}
\def\R{\mathbb{R}}
\def\T{\mathbb{T}}
\def\x{\mathbf{x}}
\def\y{\mathbf{y}}
\def\z{\mathbf{z}}
\def\w{\mathbf{w}}
\def\e{\mathbf{e}}

\begin{document}
\title{Generalized contact geometry and T-duality}


\author[VCU]{Marco Aldi}
\ead{maldi2@vcu.edu}
\address[VCU]{Department of Mathematics and Applied Mathematics, Virginia Commonwealth University, 1015 Floyd Avenue, Richmond VA 23284, United States}

\author[UNM]{Daniele Grandini \corref{cor2}\fnref{fn2,fn3}}
\ead{grandini.math@gmail.com}
\address[UNM]{Department of Mathematics and Statistics, University of New Mexico, Albuquerque NM 87131, United States}
\cortext[cor2]{Corresponding author}
\fntext[fn2]{Present address: Department of Mathematics and Applied Mathematics, Virginia Commonwealth University, 1015 Floyd Avenue, Richmond VA 23284, United States}
\fntext[fn3]{Phone: +1-(804)827-5277  }

\begin{abstract}
We study generalized almost contact structures on odd-dimensional manifolds. We introduce a notion of integrability and show that the class of these structures is closed under symmetries of the Courant-Dorfman bracket, including T-duality. We define a notion of geometric type for generalized almost contact structures, and study its behavior under T-duality.
\end{abstract}
\maketitle
\section{Introduction}
Generalized (almost) complex structures (see \citep{Gualtieri} for a systematic exposition and references) on even dimensional real manifolds  have received considerable attention from both physicists and mathematicians. While it is natural to look for odd-dimensional analogues of generalized complex geometry, the subject is still in its infancy. Generalized almost contact structures on a manifold $M$ of odd real dimension were first defined  in \citep{Poon-Wade} as a particular case of  generalized $F$-structures, introduced in \cite{Vaisman}.  In \cite{Poon-Wade}, generalized almost contact structures are equivalence classes of triples $(\Phi,F,\eta)$ (called here {\it Poon-Wade triples}) consisting of an endomorphism $\Phi$ of the generalized tangent bundle $\T M=TM\oplus T^*M$, a global vector field $F$ (playing the role of the Reeb vector field in ordinary contact geometry) and a global $1$-form $\eta$, subject to suitable axioms.

In the present paper, we define generalized almost contact structures as pairs $(E,L)$ of subbundles $E\subseteq \T M$, $L\subseteq \T M\otimes \C$ satisfying certain nondegeneracy and isotropy conditions. In particular, any Poon-Wade triple $(\Phi,F,\eta)$ yields a generalized almost contact structure $(E,L)$ according to our definition, where $E={\rm span}(F,\eta)$ and $L$ is the $\sqrt{-1}$-eigenbundle of $\Phi$.
We also introduce a local invariant of a generalized almost contact structure $(E,L)$, called the {\it geometric type}, which is a natural analogue of the type of a generalized almost complex structure, and we prove that the geometric type characterizes (up to isomorphism) the pairs $(E,L)$ coming from Poon-Wade triples. Moreover, almost contact structures and almost cosymplectic structures are characterized by having extremal geometric types.
As it turns out, every generalized almost structure $(E,L)$ on $M$ can be lifted (not uniquely) to a generalized almost complex structure on the cone $C(M):=M\times (0,\infty)$. This construction singles out a class of generalized contact structures, namely those that can be lifted to a generalized complex structure. We call such generalized almost contact structures \emph{normal}. This notion is modeled after the classical case of normal almost contact structures.
One of our main results is an intrinsic characterization of normal generalized almost structures $(E,L)$ in terms of the geometry of the subbundles $E$ and $L$. Inspired by \cite{Poon-Wade}, we give a notion of \emph{integrability} and \emph{strong integrability} for a generalized almost contact structure $(E,L)$, in terms of certain Dirac structures constructed out $E$ and $L$. We prove that $(E,L)$ if and only if it strongly integrable and $E$ admits an isotropic frame $\e_1,\e_2$ such that $[\e_1,\e_2]=-[\e_2,\e_1]$ is the orthogonal projection onto $E^{\bot}$ of an exact 1-form.
We also complete the result in \cite{Poon-Wade} by providing a full characterization of normal almost contact structures and cosymplectic structures in terms of generalized geometry.
In the second part of the paper we give an alternate description of generalized almost contact structures in the language of spinors. While generalized almost complex structures can locally be encoded by a single pure spinor, we show that a generalized almost contact structure $(E,L)$ can be locally described by a pair of pure spinors $(\rho_1,\rho_2)$ that are intertwined by the Clifford action of the bundle $E$. We call such pairs of pure spinors \emph{mixed pairs}. Our main result in this direction is that normality, integrability and strong integrability of generalized almost contact structures can be fully described in terms of mixed pairs.
The language of mixed pair is particularly convenient in order to describe the behaviour of generalized almost contact structures under T-duality. In particular, we are able to complete the result of \cite{Gualtieri-Cavalcanti} in the odd-dimensional setting providing an explicit relation between the geometric types of T-dual generalized almost contact structures. Throughout the paper we provide several examples, including a new family of generalized almost contact structures on the sphere $S^3$, for which the geometric type is generically not constant.

\section{Preliminaries on Generalized Geometry}
The \emph{generalized tangent bundle} of a real smooth manifold $M$ of dimension $m$ is the vector bundle $\T M:=TM\oplus T^*M$. It comes equipped with
\begin{itemize}
  \item an \emph{inner product}, defined by
  $$\langle X+\alpha, Y+\beta\rangle:=2^{-1}(\alpha(Y)+\beta(X)),\quad X,Y\in \Gamma(TM),\ \alpha, \beta\in \Gamma(T^*M)$$
 (which is $C^{\infty}(M)$-bilinear, symmetric, nondegenerate and with signature $(n,n)$);
  \item an $\R$-bilinear map $[\ ,\ ]: \Gamma(\T M)\times \Gamma(\T M)\rightarrow \Gamma(\T M)$ called the \emph{Dorfman bracket}, given by
  $$[X+\alpha, Y+\beta]:=[X,Y]+{\mathcal L}_X\beta-i_Yd\alpha ,\quad X,Y\in \Gamma(TM),\ \alpha, \beta\in \Gamma(T^*M).$$
\end{itemize}
(Sections of $\T M$ are denoted by $\x, \y$, etc. unless their (co)tangent components need to be specified.) Let $a: \T M\rightarrow TM$ be the obvious projection. The quadruple $(\T M, \langle\ ,\ \rangle, [\ ,\ ],a)$ satisfies the axioms of the \emph{Courant algebroids}, i.e.\ \begin{itemize}
  \item $a(\x)\left(\langle\y,\z\rangle\right)=\langle[\x,\y],\z\rangle+\langle\y,[\x,\z]\rangle$,
  \item $[\x,[\y,\z]]=[[\x,\y],\z]+[\y,[\x,\z]]$,
  \item $[\x,\y]+[\y,\x]=2d\langle\x,\y\rangle$,
\end{itemize}
for all $\x,\y,\z\in \Gamma(\T M)$. If $H$ is a closed three-form, we define the \emph{twisted Dorfman bracket} as
$$[\x,\y]_H=[\x,\y]-i_{a(\x)}i_{a(\y)}H.$$
The quadruple $(\T M, \langle\ ,\ \rangle, [\ ,\ ]_H,a)$ satisfies the axioms of Courant algebroids as well. A \emph{symmetry} of the generalized tangent bundle is a bundle morphism $F: \T M \rightarrow \T M$ such that
$$\langle F\x, F\y\rangle=\langle\x,\y\rangle,\quad [F\x , F\y]=F[\x,\y].$$
The group of all symmetries is the semidirect product
$${\rm Diff}(M)\ltimes \Omega^2_{cl}(M)$$
where a diffeomorphism $f:M\rightarrow M$ acts via its \emph{generalized Jacobian}
$$\T f:\T M\rightarrow \T M, \quad X_p+\alpha_p\longmapsto (T_pf)(X_p)+\left(T_{f(p)}f^{-1}\right)^*(\alpha_p),$$
and a closed 2-form $\omega$ acts via the \emph{gauge shift} (also called $B$-\emph{field} or $S$-\emph{field})
$$e^{\omega}:\T M\rightarrow \T M$$
$$X+\alpha\mapsto X+\alpha+i_X\omega\,.$$
\ \\
A \emph{generalized almost complex structure} on (a necessarily even dimensional) manifold $M$ is a base-fixing bundle morphism $J:\T M\rightarrow \T M$ such that $J^2=-{\rm Id}$ and $J^*=-J$.
A generalized almost complex structure $J$ is called \emph{integrable} (or a \emph{generalized complex structure}) if for all sections $\x$ and $\y$ we have
$$[J\x,J\y]-[\x,\y]-J([J\x,\y]+[\x,J\y])=0\,.$$
Generalized almost complex structures are in one-to-one correspondence with complex subbundles $L\subset\T M\otimes \C$ that are \emph{maximal isotropic} (i.e. the restriction of $\langle\ ,\ \rangle$ to $L$ vanishes and $L$ is of maximal rank with this property) and such that $L\cap \overline{L}=0$. The correspondence is given by
$$J\longleftrightarrow L=\left\{\x\in \T M\otimes \C:J\x=\sqrt{-1}\x\right\}\,.$$
A generalized almost complex structure $J$ is integrable if and only if $L$ is \emph{involutive}, i.e.\  $[\Gamma(L),\Gamma(L)]\subseteq \Gamma(L)$.
\section{Generalized almost contact structures}
\begin{mydef}Let $M$ be an odd dimensional manifold.
A \emph{generalized almost contact structure} on $M$ is a pair $(E,L)$, where
\begin{itemize}
\item $E\simeq M\times \R^2$ is a trivial subbundle of $\T M$ such that the restriction $\langle, \rangle_{|E}$ is nondegenerate and with signature $(1,1)$.
\item $L$ is a maximal isotropic subbundle of $E^{\bot}\otimes\C$ such that $L\cap \overline{L}=0$.
\end{itemize}
\end{mydef}
\begin{mydef}A \emph{generalized almost contact triple} (or just \emph{triple}, for short) is given by the data $(\Phi,\e_1,\e_2)$, where  $\Phi:\T M\rightarrow \T M$ is a base-fixing bundle morphism and sections $\e_1, \e_2\in \Gamma(\T M)$ such that
\begin{itemize}
\item $\langle\e_1, \e_1\rangle=\langle \e_2,\e_2\rangle=0$, $\langle\e_1, \e_2\rangle=1/2$;
  \item $\Phi^*=-\Phi$;
  \item $\Phi(\e_1)=0=\Phi(\e_2)$;
  \item $\Phi^2(\x)=-\x+2\langle\x,\e_1\rangle\e_2+2\langle\x,\e_2\rangle\e_1$.
\end{itemize}
Moreover, we say that a generalized almost contact triple is a {\it Poon-Wade triple} if  $(\e_1,\e_2)\in\Gamma(TM)\times \Gamma(T^*M)$.
Finally, two triples $(\Phi,\e_1,\e_2)$, $(\Phi',\e_1',\e_2')$ are \emph{homothetic} if $\Phi=\Phi'$ and ${\rm span}(\e_1,\e_2)={\rm span}(\e_1',\e_2')$.
\end{mydef}
\begin{mydef} Given a triple $(\Phi,\e_1,\e_2)$, let $E:={\rm span}(\e_1,\e_2)$ and let $L$ be the $\sqrt{-1}$-eigenbundle of $\Phi$. Then $(E,L)$ is a generalized almost contact structure.
In this case we say that the triple $(\Phi,\e_1,\e_2)$ represents the generalized almost contact structure $(E,L)$.
\end{mydef}
\begin{prop}\label{triples}
Any $(E,L)$ is represented by some triple, and two triples represent the same generalized almost contact structure if and only if they are homothetic. Moreover, the set of triples that represent a generalized almost contact structure $(E,L)$ is acted freely and transitively by the group ${\rm O}(E)\simeq C^{\infty}(M, {\rm O}(1,1))$.
\end{prop}
\begin{proof}
Given $(E,L)$, there exists a global isotropic frame $\e_1,\e_2$ of $E$ such that $\langle\e_1,\e_2\rangle=1/2$, and any two choices of such frame are related by an element of ${\rm O}(E)$. Moreover, there exists a unique $\Phi:\T M\rightarrow \T M$ with ${\rm Ker}(\Phi)=E$ and ${\rm Ker}(\Phi-\sqrt{-1}{\rm Id})=L$ and any triple $(\Phi,\e_1,\e_2)$ obtained in this way represents $(E,L)$.
\end{proof}

\begin{examples}\label{examplesoftriples}
\begin{itemize}
  \item[(a)] (Almost contact structures) An almost contact structure is a triple $(\phi,\xi,\eta)$ where $\phi:TM\rightarrow TM$, $\xi$ is a vector field and $\eta$ is a one-form such that $\eta(\xi)=1$, $\phi(\xi)=0=\phi^*(\eta)$ and $\phi^2(X)=-X+\eta(X)\xi$. An almost contact structure defines a Poon-Wade triple $(\Phi,\e_1,\e_2)$, where $\e_1:=\xi$, $\e_2:=\eta$,
      $$\Phi:=\left[\begin{array}{cc}\phi&0\\0&-\phi^*\end{array}\right]$$
 and the blocks correspond to the splitting $\T M=TM\oplus T^*M$.

 \item[(b)] (Almost cosymplectic structures) An almost symplectic structure is a pair $(\theta,\eta)$ where $\theta$ is a 2-form and $\eta$ is a 1-form such that
 $\theta^n\wedge\eta\neq 0$ $({\rm dim}\ M=2n+1)$. Then, $\theta$ is nondegenerate on ${\rm Ker}(\eta)$ and there is a unique vector field $\xi$ such that $$\eta(\xi)=1,\quad i_{\xi}\theta=0.$$
 Let $\phi:{\rm Ker}(\eta)\to{\rm Ann}(\xi)$ be the isomorphism defined by $\phi(X)= i_X\theta$ for all $X\in {\rm Ker}(\eta)$.
If $\e_1:=\xi$, $\e_2:=\eta$, and $\Phi$ is defined by $\Phi(\e_1)=\Phi(\e_2)=0$ and
 $$\Phi_{|E^{\bot}}:=\left[\begin{array}{cc}0&-\phi^{-1}\\\phi&0\end{array}\right]$$
 where the blocks correspond to the splitting $E^{\bot}={\rm Ker}(\eta)\oplus {\rm Ann}(\xi)$, then $(\Phi,\e_1,\e_2)$ is a Poon-Wade triple. The particular case $\theta=d\eta$ shows that contact structures are examples of generalized almost contact structures.

 \item[(c)] (Isomorphic and equivalent structures) If $(E,L)$ is a generalized almost contact structure and $F:\T M\rightarrow \T M$ is a base-fixing bundle morphism which is orthogonal (i.e. it preserves the inner product), then $(F(E), F(L))$ is a generalized almost contact structure. Two generalized almost contact structures that are related by orthogonal morphisms are called \emph{isomorphic}. Generalized almost contact structures that are related by symmetries are called \emph{equivalent}.

\item[(d)] (Products \cite{GT}) Let $M=M_1\times M_2$, $\pi_i:M\rightarrow M_i$ be the projections. If $(E, L_1)$ is a generalized almost contact structure on $M_1$ and $L_2$ is a generalized almost complex structure on $M_2$, then $(\pi_1^*E, \pi_1^*L_1\oplus\pi_2^*L_2)$ is a generalized almost contact structure on $M$. The product of two generalized almost contact manifolds is never generalized almost contact, for dimensional reasons. However, it admits generalized almost complex structures.

 \item[(e)] (Deformations) Let $(E,L)$  be a generalized almost contact structure on $M$, and let $\varepsilon: L\rightarrow\overline{L}$ be a skew bundle morphism such that ${\rm Id}_L-\overline{\varepsilon}\varepsilon$ is invertible. If $L_{\varepsilon}$ is the graph of $\varepsilon$, then $(E,L_{\varepsilon})$ is a generalized almost contact structure.
\end{itemize}
\end{examples}

\section{The geometric type}
\begin{mydef}Given a generalized almost contact structure $(E,L)$, its \emph{geometric type} at $x\in M$ is the pair $(p_E(x),t_L(x))$, where
$$p_E(x):={\rm dim}(a(E_x)), \quad t_L(x):={\rm codim}_{\C}(a(L_x)).$$\end{mydef}
It is easy to see that the geometric type is invariant under symmetries.

\begin{prop}
Let ${\rm dim}(M)=2n+1$. Only two cases can occur:
\begin{itemize}
  \item $p_E(x)=1$ and $1\leq t_L(x)\leq n+1$;\\
  \item $p_E(x)=2$ and $1\leq t_L(x)\leq n$.
\end{itemize}
\end{prop}

\noindent
\begin{proof}We have
\[
{\rm dim} (a(E^{\bot}))=2(2n+1-t_L)-{\rm dim}_{\C}(a(L)\cap a(\overline{L}))\,.
\]
Since $E$ cannot be a subbundle of $T^*M$, either $p_E=1$ or $p_E=2$. Let $p_E=2$. Then $E$ is generated by $\e_i=X_i+\alpha_i$, where $i=1,2$ and $X_1, X_2$ are linearly independent. Then for each tangent vector $Y$ there is a cotangent vector $\beta$ such that $$\beta(X_1)=-\alpha_1(Y),\quad \beta(X_2)=-\alpha_2(Y)\,.$$
In other words, $a(E^{\bot})=TM$, hence
\[
2t_L=2n+1-{\rm dim}_{\C}(a(L)\cap a(\overline{L}))
\]
and $1\leq t_L\leq n$. If $p_E=1$, then $E$ is generated by $\e_1=X+\alpha\,, \e_2=\beta$ with $X\neq 0\neq \beta$.
It follows that $a(E^{\bot})={\rm Ker}(\beta)$ and
\[
2t_L=2n+2-{\rm dim}_{\C}(a(L)\cap a(\overline{L}))
\]
which implies $1\leq t_L\leq n+1$.\end{proof}

\begin{examples}\label{examplestype}
\begin{itemize}
\item[(a)] If $(E,L)$ is represented by an almost cosymplectic structure, then $p_E=1$, $t_L=1$.

\item[(b)] If $(E,L)$ is represented by an almost contact structure, then $p_E=1$, $t_L=n+1$.

\item[(c)] Consider $S^1$ equipped with the obvious constant generalized almost contact structure of geometric type $(1,1)$. Also, let $C$ be a cubic in the complex projective plane $\C\mathbb{P}^2$. Then there is a generalized complex structure on $\C\mathbb{P}^2$ whose type equals $2$ along the cubic and $0$ elsewhere \citep{Gualtieri}. It follows that the Cartesian product $S^1\times \C\mathbb{P}^2$ is equipped with the product structure of \ref{examplesoftriples}(4), which is a generalized almost contact structure $(E,L)$ with $p_E=1$ and
\[
t_L=\left\{\begin{array}{ll} 3 & \mbox{ on }S^1\times C\,,\\
                                     1 & \mbox{ elsewhere\,. } \end{array}\right.
\]

\item[(d)] The previous example has the following generalization. Let $\pi:M^{2n+1}\rightarrow N^{2n}$ be a principal bundle, let $TM=V\oplus {\rm Hor}$ be a splitting into vertical and horizontal subbundles and let
\[
E=V\oplus {\rm Ann}({\rm Hor})\subset \mathbb{T}M\,.
\]
Given a generalized almost complex structure $L\subset \mathbb{T}N\otimes \C$ define
\[
(\pi^\star L)_p:=\left\{v+\pi^*\lambda\in ({\rm Hor}_p \oplus T^*M) \otimes \C\,|\, \pi_*v+\lambda\in L_{\pi(p)}\right\}\,,
\]
for each $p\in M$.
Then, $(E,\pi^\star L)$ is a generalized almost contact structure on $M$ with $p_E=1$. In particular, $S^5$ (a circle bundle over $\C\mathbb{P}^2$) admits generalized almost contact structures with nonconstant $t_L$.
  \item[(e)] Consider a {\it triple almost contact structure}, i.e.\ a triple of almost contact structures $\{(\phi_a,\xi_a, \eta_a)\}_{a=1,2,3}$ such that
\[
\eta_a(\xi_b)=\delta_{ab},\quad\phi_a(\xi_b)=\epsilon_{abc}\xi_c,\quad \phi^*_a\eta_b=-\epsilon_{abc}\eta_c\, ,
\]
\[
\phi_a\phi_b=-\delta_{ab}{\rm Id}+\xi_a\otimes \eta_b+\epsilon_{abc}\phi_c\,.
\]
If
\[\e_1=2^{-1/2}(\xi_1+\eta_2),\quad \e_2=2^{-1/2}(\xi_2+\eta_1)
\]
and $\Phi_0:\T M\to \T M$ is defined by
\begin{align}
\Phi_0(X)&=\phi_3(X)-\eta_1(X)\xi_2+\eta_2(X)\xi_1+\nonumber\\
&+2^{-1/2}\left(\eta_1(X)\xi_3-\eta_3(X)\xi_1-\eta_2(X)\eta_3+\eta_3(X)\eta_2\right),\nonumber\\
\Phi_0(\alpha)&=-\phi_3^*(\alpha)-\alpha(\xi_1)\eta_2+\alpha(\xi_2)\eta_1+\nonumber\\
              &+2^{-1/2}\left(\alpha(\xi_1)\eta_3-\alpha(\xi_3)\eta_1-\alpha(\xi_2)\xi_3+\alpha(\xi_3)\xi_2\right),\nonumber\end{align}
for any vector field $X$ and any 1-form $\alpha$, then $(\Phi_0,\e_1,\e_2)$ is a generalized almost contact triple such that $p_E=2$ and $t_L=n$.
\end{itemize}
\end{examples}
\begin{prop}
Let $\{(\phi_a,\xi_a, \eta_a)\}_{a=1,2,3}$ be a triple almost contact structure and let $S$ be spanned by $\xi_a$ and $\eta_b$ for $a,b=1,2,3$. Let $\e_1=2^{-1/2}(\xi_1+\eta_2), \e_2=2^{-1/2}(\xi_2+\eta_1)$ and let $\mathcal F$ be the set of all $\Phi:\T M\to \T M$ satisfying the following properties:
\begin{itemize}
  \item[i)] $(\Phi,\e_1,\e_2)$ is a generalized almost contact triple;
  \item[ii)] $\Phi(S)\subset S$;
  \item[iii)] the restriction of $\Phi$ to $S^{\bot}$ is of the form $$\left[\begin{array}{cc}\phi_3&0\\0&-\phi^*_3\end{array}\right];$$
  \item[iv)] $\langle \Phi\xi_a,\eta_b\rangle=-\langle \Phi\xi_b,\eta_a\rangle$ for all $a,b\in \{1,2,3\}$.
\end{itemize}
Then $${\mathcal F}=\{\Phi_0, \sigma\Phi_0, \tau\Phi_0, \sigma\tau\Phi_0\}$$
where $\Phi_0$ is as in Example \ref{examplestype}(e) and $\sigma, \tau: \T M\to \T M$ are the involutions
$$\sigma:\quad \xi_i\leftrightarrow -\xi_i, \eta_i\leftrightarrow -\eta_i,$$
$$\tau:\quad \xi_1\leftrightarrow \xi_2, \eta_1\leftrightarrow \eta_2, \xi_3\leftrightarrow -\eta_3.$$
\end{prop}

\noindent
\begin{proof} Any morphism $\Phi$ satisfying the third condition must be such that, for all $X\in \Gamma(TM)$
$$\Phi\left(X-\sum_i\eta_i(X)\xi_i\right)=\phi_3\left(X-\sum_i\eta_i(X)\xi_i\right)=\phi_3(X)-\eta_1(X)\xi_2+\eta_2(X)\xi_1\,,$$
so that
$$\Phi(X)=\phi_3(X)-\eta_1(X)\xi_2+\eta_2(X)\xi_1+\sum_i\eta_i(X)\Phi(\xi_i)\,.$$
Analogously, for all $\alpha\in\Gamma(T^*M)$,
$$\Phi(\alpha)=-\phi^*_3(\alpha)-\alpha(\xi_1)\eta_2-\alpha(\xi_2)\eta_1+\sum_i\alpha(\xi_i)\Phi(\eta_i)\,.$$
Now, write
$$\Phi(\xi_i)=\sum_ja_{ij}\xi_j+b_{ij}\eta_j,\quad \Phi(\eta_i)=\sum_jc_{ij}\xi_j+d_{ij}\eta_j\,.$$
From $\Phi^*=-\Phi$ we get $a_{ij}=-d_{ji}$, $b_{ij}=-b_{ji}$, $c_{ij}=-c_{ji}$, and since $\Phi(\xi_1+\eta_2)=\Phi(\xi_2+\eta_1)=0$
we obtain
$$a_{11}=-a_{22}=c_{12}=b_{21},\quad a_{12}=a_{21}=0\,,$$
$$a_{13}=c_{32},\quad a_{23}=c_{31},\quad a_{31}=b_{23},\quad a_{32}=b_{13}\,.$$
Therefore, $\Phi$ is uniquely determined by the coefficients $a_{ij}$. From
\[
\Phi^2(\xi_1)=2^{-1}(-\xi_1+\eta_2)\,,\quad \Phi^2(\xi_2)=2^{-1}(-\xi_2+\eta_1)\,, \quad\Phi^2(\xi_3)=-\xi_3\,
\]
and $\langle \Phi\xi_a,\eta_b\rangle=-\langle \Phi\xi_b,\eta_a\rangle$ we obtain the conditions $$a_{ij}=-a_{ji},$$
$$2a_{13}^2+2a_{23}^2=1\quad\mbox{ and }\quad a_{13}a_{23}=0,$$
which yield the required four solutions. The morphism $\Phi_0$ corresponds to choosing $a_{23}=0$, $a_{13}>0$.
\end{proof}
\begin{theorem}\label{pE=1} Let $(E,L)$ be a generalized almost contact structure on $M$, such that ${\rm dim}\ M=2n+1$ and $p_E=1$. Then
\begin{itemize}
  \item[i)] $(E,L)$ is isomorphic to a generalized almost contact structure represented by a Poon-Wade triple;
  \item[ii)] if $t_L=1$, then $(E,L)$ is isomorphic to a generalized almost contact structure represented by an almost cosymplectic structure;
  \item[iii)] if $t_L=n+1$, then $(E,L)$ is isomorphic to a generalized almost contact structure represented by an almost contact structure.
\end{itemize}
\end{theorem}
\begin{proof}Proof of (i): $p_E=1$ if and only if $E$ is generated by $\e_1=X+\alpha$, $\e_2=\beta$. Set $\Omega=\alpha\wedge\beta$, so that $e^{\Omega}\e_1=X$ and $e^{\Omega}\e_2=\beta$. Hence, $(e^{\Omega}\e_1,e^{\Omega}\e_2, e^{\Omega}\circ\Phi\circ e^{-\Omega})$ is a Poon-Wade triple.\\
Proof of (ii): from (i), we can assume up to isomorphism that $E$ is generated by $X, \beta$,  with $\beta(X)=1$, so that
$$E={\rm Ker}(\beta)\oplus {\rm Ann}(X)=:{\mathcal K}\oplus {\mathcal A}\,.$$
Let
$$\Phi=\left[\begin{array}{cc}A&B\\C&-A^*\end{array}\right]\,,$$
where the blocks correspond to the splitting $\T M=TM\oplus T^*M$. The condition
$p_E=1=t_L$ implies $a(L)=a(\overline{L})$, which means that for all $Y+\gamma\in \Gamma(E^{\bot})$ there exists $\widetilde{\gamma}\in \Gamma(T^*M)$ such that
$$Y+\widetilde{\gamma}\in \Gamma(E^{\bot})\,, \quad 2A(Y)+B(\gamma+\widetilde{\gamma})=0\,.$$
This implies the decomposition
$$Y+\gamma=\left(Y+2^{-1}(\gamma+\widetilde{\gamma})\right)+2^{-1}(\gamma-\widetilde{\gamma})$$
and the splitting
$$E^{\bot}=(\Phi(E^{\bot}\cap T^*M))+(E^{\bot}\cap T^*M)=\Phi({\mathcal A})\oplus {\mathcal A}\,.$$
Moreover, $\Phi({\mathcal A}) \cap {\mathcal A}=0$ implies that the map $B_{|{\mathcal A}}:{\mathcal A}\rightarrow TM$ is injective.
Since $B({\mathcal A})={\mathcal K}$, we have $$\Gamma(\Phi({\mathcal A}))=\{Y-A^*B^{-1}Y|Y\in \Gamma({\mathcal K}) \}.$$
If $Y\in \Gamma(TM)$, then $\langle AY, \beta\rangle=\langle\Phi Y,\beta\rangle=0$, hence $AY\in \Gamma({\mathcal K})$.
Moreover, for all $Y\in\Gamma({\mathcal K})$ we have
$$A^*B^{-1}Y=B^{-1}AY.$$
Define the 2-form on $M$:
$$\Omega(Z,W):=\langle B^{-1}AZ,W\rangle-\langle B^{-1}AW,Z\rangle.$$
For all vector fields $W$ and for all $Y\in\Gamma({\mathcal K})$
\[
\Omega(X,W)=0\,,\quad \Omega(Y,W)=2\langle A^*B^{-1}Y, W\rangle\, .
\]
Therefore, $e^{-\Omega}{\mathcal K}=\Phi({\mathcal A})$ and the bundle morphism $\phi:=e^{\Omega}\circ \Phi\circ e^{-\Omega}$ is such that
\[
\phi({\rm Ker}(\beta))={\rm Ann}(X),\quad \phi({\rm Ann}(X))={\rm Ker}(\beta)\,.
\]Proof of (iii): in this case we have $a(L)\cap a(\overline{L})=0$. Since $B({\mathcal A})=a(\Phi({\mathcal A}))\subset a(L)\cap a(\overline{L})$, then $B=0$.
Notice that $\Omega=2^{-1}CA$ is a 2-form since $\Omega^*=2^{-1}(CA)^*=-2^{-1}A^*C=-\Omega$. Moreover, $i_X\Omega=0$ and
\[
\left[\begin{array}{cc}I&0\\ \Omega&I\end{array}\right]\left[\begin{array}{cc}A&0\\C&-A^*\end{array}\right]\left[\begin{array}{cc}I&0\\-\Omega&I\end{array}\right]=\left[\begin{array}{cc}A&0\\0&-A^*\end{array}\right]\,.
\]
Therefore, $(e^{\Omega}E,e^{\Omega}L)$ is an almost contact structure.
\end{proof}

\section{Generalized geometry of the cone}\label{Cone}
Let $\mathbb R_t^+$ be the positive real axis with coordinate $t$. Moreover, let $C(M):=M\times \R^+_t$ be the \emph{cone} over $M$ and let $\pi_1,\pi_2$ the projections onto the two factors. The generalized tangent bundle of the cone is
$$\T C(M)=\pi_1^*\T M \oplus \pi_2^* \T^*\R^+_t\,.$$
The multiplication by a positive real number $t_0$ induces a symmetry $F_{t_0}: \T C(M)\rightarrow \T C(M)$. Upon the identification $M\equiv M\times \{1\}$, we have the inclusion $i:M\hookrightarrow C(M)$. Let $${\mathcal E}'(M):=i^*\T C(M), \quad {\mathcal E}_x'(M)=\T_xM\oplus {\rm span}(\partial_t,dt)\,.$$
A section of ${\mathcal E}'(M)$ is of the form $\x+f\partial_t+gdt$, where $\x\in \Gamma(\T M)$ and $f,g\in C^{\infty}(M)$. The bundle ${\mathcal E}'(M)$ inherits from $\T C(M)$ the inner product
$$\langle \mathbf X_1,\mathbf X_2\rangle=\langle\x_1,\x_2\rangle+2^{-1}(f_1g_2+f_2g_1)\,,$$
and the bracket
\begin{eqnarray}[ \mathbf X_1,\mathbf X_2]_0&=&[\x_1,\x_2]+(a(\x_1)(f_2)-a(\x_2)(f_1))\partial_t+\nonumber\\
                                                   &&+(a(\x_1)(g_2)-a(\x_2)(g_1))dt+2^{-1}(g_2df_1+f_2dg_1)\nonumber\,.\end{eqnarray}
where $\mathbf X_i= \x_i+f_i\partial_t+g_idt$ for $i=1,2$. In particular, $[\ ,\ ]_0$ is the extension of the Dorfman bracket on $\mathbb{T}M$ such that for all $\x,\y\in\Gamma(\mathbb{T}M)$,
\begin{align}
[\x,\y]_0&= [\x,\y],\ &\ \ [\x,dt]_0&=[dt,\x]_0=0,\nonumber \\
[\x,\partial_t]_0&= [\partial_t,\x]_0=0,\   &\ \ [dt,\partial_t]_0&=[\partial_t,dt]_0=0.\nonumber
\end{align}
(The subscript ``0" will be omitted in the rest of the paper.)
\begin{mydef}We say that a quadruple $(\Phi,\e_1,\e_2, \lambda)$, where $\Phi:\T M\rightarrow \T M$ is a base-fixing bundle morphism, $\e_1,\e_2\in \Gamma(\T M)$ and $\lambda\in C^{\infty}(M)$ is a {\it Sekiya quadruple} on $M$ if
\begin{itemize}
\item $\langle\e_1, \e_1\rangle=\langle \e_2,\e_2\rangle=0$, $\langle\e_1, \e_2\rangle=1/2$;
   \item $\Phi^*=-\Phi$;
  \item $\Phi(\e_1)=\lambda\e_1$, $\Phi(\e_2)=-\lambda\e_2$ ;
  \item $\Phi^2\x=-\x+2(1+\lambda^2)(\langle\x,\e_2\rangle\e_1+\langle\x,\e_1\rangle\e_2)$.
\end{itemize}\end{mydef}
\begin{rem} Observe that there is a canonical bijection between the set of $\R^+$-invariant generalized almost complex structures on $C(M)$ and the set of $S^1$-invariant generalized almost complex structures on $M\times S^1$, through exponentiation. Moreover, any $\R^+$-invariant generalized almost complex structure $L$ on $C(M)$ is uniquely determined by its restriction to $M$. Therefore, we have a canonical bijection between $\R^+$-invariant generalized almost complex structures and base-preserving bundle morphisms $J:{\mathcal E}'(M)\rightarrow {\mathcal E}'(M)$ such that $J^2=-{\rm Id}$, $J^*=-J$.\end{rem}

\begin{prop}\label{Bijection-Sekiya} Let $J:{\mathcal E}'(M)\rightarrow {\mathcal E}'(M)$ be a base-preserving bundle morphism such that $J^2=-{\rm Id}$, $J^*=-J$. Then, $(\Phi,\e_1,\e_2, \lambda)$ where
\begin{align}
\lambda&=2\langle J(dt),\partial_t\rangle;\nonumber\\
\e_1&=(1+\lambda^2)^{-1/2}\left(J(\partial_t)+\lambda\partial_t\right);\nonumber\\
\e_2&=(1+\lambda^2)^{-1/2}\left(J(dt)-\lambda dt\right)\nonumber\end{align}
and
$$\Phi(\x)=J(\x)+2\langle\partial_t,J(\x)\rangle dt + 2\langle dt, J(\x)\rangle \partial_t$$
is a Sekiya quadruple. Viceversa, any Sekiya quadruple defines a unique $J$ with the above properties.
 \end{prop}

\begin{proof} It easy to see that $(\Phi,\e_1,\e_2, \lambda)$ defined above is a Sekiya quadruple.
Conversely, starting with a Sekiya quadruple $(\Phi,\e_1,\e_2, \lambda)$, let $J:{\mathcal E}'(M)\rightarrow {\mathcal E}'(M)$, be defined by setting $J(\x)=\Phi(\x)$  for all  $\x\in \Gamma(\T M)$ such that $\langle \x,\e_1\rangle=\langle \x,\e_2\rangle=0$, and
\begin{align}
J(\e_1)&=\lambda\e_1-(1+\lambda^2)^{1/2}\partial_t;\nonumber\\
J(\e_2)&=-\lambda\e_2-(1+\lambda^2)^{1/2}dt;\nonumber\\
J(\partial_t)&=(1+\lambda^2)^{1/2}\e_1-\lambda\partial_t;\nonumber\\
J(dt)&=(1+\lambda^2)^{1/2}\e_2+\lambda dt.\nonumber
\end{align}
Then, $J$ satisfies $J^2=-{\rm Id}$, $J^*=-J$. It is easy to see that these constructions are inverse of each other.
\end{proof}
\begin{rem}\label{Sek} In particular, under the bijection of Proposition \ref{Bijection-Sekiya}, the set of generalized almost contact triples (viewed as Sekiya quadruples with $\lambda=0$) corresponds to the set ${\rm Sek}_0(M)$ of $\R^+$-invariant generalized almost complex structures on $C(M)$ such that $J({\rm span}(\partial_t, dt))\subset \T M$. Moreover, the $C^{\infty}(M, {\rm O}(1,1))$-action on triples lifts to an action on ${\rm Sek}_0(M)$.
\end{rem}
\begin{mydef}
We say that a generalized almost contact triple is associated to $J\in {\rm Sek}_{0}(M)$ if they correspond to each other via  Proposition \ref{Bijection-Sekiya}.
Moreover, a generalized almost contact structure $(E,L)$ on $M$ is represented by $J\in {\rm Sek}_0(M)$ if the triple associated to $J$ represents $(E,L)$.\end{mydef}
\begin{theorem}\label{J_0}
Every generalized almost contact structure $(E,L)$ on $M$ is represented by an element of ${\rm Sek}_0(M)$. Moreover, two elements $J_1, J_2\in {\rm Sek}_0(M)$ represent the same generalized almost contact structure $(E,L)$ if and only if they belong to the same  ${\rm O}(E)$-orbit.
\end{theorem}
\begin{proof}
The first statement follows directly from Propositions \ref{triples}, \ref{Bijection-Sekiya} and Remark \ref{Sek}. Now, let $J, \tilde J\in {\rm Sek}_0(M)$ and let $(\Phi, \e_1, \e_2)$ and $(\tilde \Phi, \tilde \e_1, \tilde \e_2)$ be the associated triples. Then, $J$ and  $\tilde J$ represent the same generalized almost contact structure $(E,L)$ if and only if
$\Phi=\tilde \Phi$ and $\tilde \e_i= R \cdot \e_i$ with $R\in {\rm O}(E)$. Then, $J=\tilde JR$.

\end{proof}

\begin{prop}
Let $(E,L)$ be represented by $J\in{\rm Sek}_0(M)$ of type $t_J$. Then
\begin{itemize}
  \item[i)] $0\le t_L-t_J\le 1$;
  \item[ii)] $t_J=t_L$ if and only if $a(J(dt))\in a(L)$;
  \item[iii)] if $p_E=1$, then  $t_J=t_L$ if and only if $J(dt))$ is a 1-form on $M$.
\end{itemize}

\end{prop}
\begin{proof} Let $\widetilde{a}:\T C(M)\rightarrow TC(M)$ be the obvious projection. The maximal isotropic subbundle $L_J\subseteq\T C(M)$ that defines $J$ is spanned by $L$, $\e_1+\sqrt{-1}\partial_t$ and $\e_2+\sqrt{-1}dt$, where ${\rm span}(\e_1,\e_2)=E$. Therefore,
\[
\dim \tilde a(L_J) = \dim a (L)+ 1 +\dim a(\C \e_2) - \dim (a(L)\cap\dim a(\C \e_2))\,,
\]
from which it follows that
\[
t_J={\rm dim}\ M+1-\dim \tilde a (L_J) = t_L - \dim a (\C \e_2) +\dim (a(L)\cap\dim a(\C \e_2))\,,
\]
which proves $(i)$ and (ii). Finally, assume  $p_E=1$. Then
$${\rm dim}\left(a(E)\cap a(E^{\bot})\right)=2p_E-2$$
implies $a(E\otimes\C)\cap a(L)=0$ and (iii).
\end{proof}

\begin{example}\label{Sphere}(New generalized almost contact structures on $S^3$ and $\R\mathbb{P}^3$)
Consider the standard identification of $C(S^3)$ with $\mathbb{R}^4\setminus \{0\}$. The Euler vector field $\partial_t$ is given by
$$\partial_t=x_1 \frac{\partial}{\partial x_1}+x_2\frac{\partial}{\partial x_2}+x_3 \frac{\partial}{\partial x_3}+x_4 \frac{\partial}{\partial x_4}\,.$$
The action of ${\rm Sp}(1)$ on $\mathbb{R}^4$ is generated by the vector fields
$$V_1:=x_2 \frac{\partial}{\partial x_1}-x_1\frac{\partial}{\partial x_2}+x_4 \frac{\partial}{\partial x_3}-x_3 \frac{\partial}{\partial x_4}\,,$$
$$V_2:=x_3 \frac{\partial}{\partial x_1}-x_4\frac{\partial}{\partial x_2}-x_1 \frac{\partial}{\partial x_3}+x_2 \frac{\partial}{\partial x_4}\,,$$
$$V_3:=x_4 \frac{\partial}{\partial x_1}+x_3\frac{\partial}{\partial x_2}-x_2 \frac{\partial}{\partial x_3}-x_1 \frac{\partial}{\partial x_4}\,.$$
In particular, $\left\{\partial_t,V_1,V_2,V_3\right\}$ is a global frame of $TM$. Let $\left\{dt ,\nu_1,\nu_2,\nu_3\right\}$ denote the dual frame. Consider $j\in {\rm Sek}_0(S^3)$ defined by
\[
j(V_1)=\partial_t\,,\quad j(V_2)=V_3\,,\quad j^*(\nu_1)=-dt\,, \quad j^*(\nu_2)=-\nu_3.
\]
Given $f,g\in C^{\infty}(S^3)$, consider the $\R^+$-invariant bivector
\[
\Lambda=f\cdot\left(-V_1\wedge V_3+V_2\wedge \partial_t\right)+g\cdot \left(V_1\wedge V_2+V_3\wedge\partial_t\right)\,.
\]
Since $j\Lambda = \Lambda j^*$,
\[
J=
\left[\begin{array}{cc}
j&\Lambda\\
0 &-j^*\end{array}\right]
\]
is also in ${\rm Sek}_0(S^3)$. Moreover,
\begin{align}\label{e1e2}
\e_1:=J\left(\partial_t\right)=-V_1; \quad \e_2:=J\left(dt\right)=-\nu_1-fV_2-gV_3,\end{align}
which implies that the generalized almost contact structure $(E,L)$ on $S^3$ represented by $J$ has geometric type
\[
(p_E(x),t_L(x))=\left\{\begin{array}{cl}(1,2)&\mbox{ if } f(x)=0=g(x)\\\\
                                    (2,1)&\mbox{ otherwise\,.}\end{array}\right.
\]
If $f$ and $g$ are invariant under the antipodal involution, then $(E,L)$ descends to a generalized almost contact structure on $\R\mathbb{P}^3$, of the same geometric type.
\end{example}
\section{Normality and integrability}
\begin{mydef}A generalized almost contact structure $(E,L)$ is called \emph{normal} if there exists $J\in {\rm Sek}_0(M)$ which is integrable and represents $(E,L)$. Moreover, a triple $(\Phi, \e_1,\e_2)$ is \emph{normal} if the generalized almost complex structure associated it is integrable.\end{mydef}

\begin{prop}\label{normality} A triple $(\Phi, \e_1,\e_2)$ is normal if and only if the following conditions hold:
\begin{itemize}
  \item[i)] for all $\x,\y\in\Gamma(E^{\bot})$,
  $$[\Phi\x,\Phi\y]-[\x,\y]-\Phi([\Phi\x,\y]+[\x,\Phi\y])=0;$$
  \item[ii)] for all $\x\in\Gamma(E^{\bot})$,
  $\Phi[\x,\e_i]=[\Phi\x,\e_i]$ ($i=1,2$),
  \item[iii)] $[\e_1,\e_2]=0$,
\end{itemize}
where $E={\rm span}(\e_1,\e_2)$.
\end{prop}
\begin{proof} Let $J\in {\rm Sek}_0(M)$ associated to $(\Phi,\e_1,\e_2)$ and let $N$ denote the Nijenhuis operator of $J$, i.e.\
\[
N(\mathbf{X},\mathbf{Y})=[J\mathbf{X},J\mathbf{Y}]-[\mathbf{X},\mathbf{Y}]-J([J\mathbf{X},\mathbf{Y}]+[\mathbf{X},J\mathbf{Y}])
\]
for all $\mathbf X,\mathbf Y\in \Gamma ({\mathcal E}'(M))$. Notice that $N$ is defines a skew-symmetric bundle map $N: {\mathcal E}'(M)\otimes {\mathcal E}'(M)\rightarrow {\mathcal E}'(M)$ and that $N(\mathbf{X},J\mathbf{Y})=-J(N(\mathbf{X},\mathbf{Y}))$. Then $N=0$ if and only if the following conditions are satisfied:
$$N(\x,\y)=0,\quad N(\x,\e_1)=0,\quad N(\x,\e_2)=0,\quad N(\e_1, \e_2)=0$$
 for any $\mathbf{x},\mathbf{y}\in \Gamma (E^{\bot})$.\\ Assume now $N=0$. From $0=N(\e_1,\e_2)=-[\e_1,\e_2]$ follows (iii). Moreover, from (iii) we have $[\x,\e_i]\in\Gamma(E^{\bot})$ for all $\x$. Hence
 $$0=N(\x,\e_i)=-[\x,\e_i]-\Phi[\Phi\x,\e_i]$$
 which yields (ii). Furthermore, from (ii) we have $[\x,\Phi\y]+[\Phi\x,\y]\in\Gamma(E^{\bot})$ for all $\x,\y$ and
 $$0=N(\x,\y)=[\Phi\x,\Phi\y]-[\x,\y]-\Phi([\Phi\x,\y]+[\x,\Phi\y])$$
whence (i). The proof of the reverse implication is analogous.
\end{proof}
\begin{rem}
Normality of triples is not preserved under homotheties. However, if two triples differ by a constant homothety, then one is normal if and only if the other is normal. Moreover, symmetries preserve the normality of triples, hence the normality of generalized almost contact structures.
\end{rem}
\begin{mydef}
A generalized almost contact structure $(E,L)$ is called \emph{integrable} if there exists a maximal isotropic subbundle $\ell$ of $E$ such that $L\oplus (\ell\otimes \C)$ is involutive.
Moreover, $(E,L)$ is strongly integrable if for any maximal isotropic subbundle $\ell$ of $E$, $L\oplus (\ell\otimes \C)$ is involutive.\end{mydef}
\begin{rem}Let $(E,L)$ be represented by a Poon-Wade triple $(\Phi,F,\eta)$. Then, if  $(\Phi,F,\eta)$ is a \emph{generalized contact structure} in the sense of \citep{Poon-Wade}, then $(E,L)$ is integrable. Moreover, $(E,L)$ is strongly integrable if and only if $(\Phi,F,\eta)$ is a \emph{strong generalized contact structure} in the sense of  \citep{Poon-Wade}.
\end{rem}

\begin{theorem}\label{normalinv}
Let $(E,L)$ be a generalized almost contact structure on $M$. The following are equivalent:
 \begin{itemize}
   \item [i)] $(E,L)$ is normal;
   \item [ii)] $(E,L)$ is strongly integrable and there exists an isotropic frame $\e_1,\e_2$ of $E$ and $f\in C^{\infty}(M)$ such that
   $$d\langle\e_1,\e_2\rangle=0,\quad [\e_1,\e_2]=df-2\langle\e_1, df\rangle\e_2-2\langle\e_2, df\rangle\e_1$$
   \item [iii)] $(E,L)$ is strongly integrable and for each isotropic frame $\e_1,\e_2$ of $E$ such that
   $d\langle\e_1,\e_2\rangle=0$, there exists $f\in C^{\infty}(M)$ such that
   $$[\e_1,\e_2]=df-2\langle\e_1, df\rangle\e_2-2\langle\e_2, df\rangle\e_1.$$
 \end{itemize}
\end{theorem}
\begin{proof}(i) implies (ii) Suppose $(E,L)$ normal, and let $(\Phi,\e_1,\e_2)$ be a normal triple that represents $(E,L)$. Let $\x\in\Gamma(E^{\bot})$, $f\in C^{\infty}(M)$. Then
\begin{align}[\x-\sqrt{-1}\Phi(\x),f\e_i]&=a\left(\x-\sqrt{-1}\Phi(\x)\right)(f)\e_i+f\left([\x,\e_i]-\sqrt{-1}\Phi[\x,\e_i]\right).\nonumber\end{align}
Since $[\x,\e_i]\in\Gamma(E^{\bot})$, we obtain $[\Gamma(L),\Gamma(\C\e_i)]\subset\Gamma(L\oplus \C\e_i)$. Also, $[\Gamma(\C\e_i),\Gamma(\C\e_i)]\subset \Gamma(\C\e_i)$ because $\C\e_i$ is a isotropic line bundle. Moreover, for any $\x,\y\in \Gamma(E^{\bot})$,
\begin{align}[\x-\sqrt{-1}\Phi(\x), \y-\sqrt{-1}\Phi(\y)]&=-\Phi\left([\x,\Phi(\y)]+[\Phi(\x),\y]-\sqrt{-1}\Phi\left([\x,\Phi(\y)]+[\Phi(\x),\y]\right)\right),\nonumber
\end{align}
which shows $[\Gamma(L),\Gamma(L)]\subset\Gamma(L)$, whence $(E,L)$ is strongly integrable. Since $[\e_1,\e_2]=0$ we obtain (ii).\\\\
(ii) implies (iii): let $\e_1,\e_2$ as in (ii) and let $\e_1',\e_2'$ be a isotropic frame of $E$ with $d\langle\e_1',\e_2'\rangle=0$. Then there exist two functions $\lambda,\mu$ such that
$\lambda\mu$ is a nonzero constant, and either
\begin{align}(\e'_1,\e'_2)=(\lambda\e_1,\mu\e_2)\quad\mbox{ or }\quad(\e'_1,\e'_2)=(\lambda\e_2,\mu\e_1).\nonumber\end{align}
Moreover, we can assume the former condition up to symmetry. Then,
\begin{align}\label{lambdamu3}[\lambda\e_1,\mu\e_2]&=dF-2\langle\e_1, dF\rangle\e_2-2\langle\e_2, dF\rangle\e_1,\end{align}
where $F=\mu\lambda(f+\ln|\lambda|)$.\\
(iii) implies (i): Let $(E,L)$ be strongly integrable, and let $\e_1,\e_2$ be an isotropic frame of $E$ normalized such that $\langle\e_1,\e_2\rangle=1/2$. Then,
$$[\e_1,\e_2]=df-2\langle\e_1, df\rangle\e_2-2\langle\e_2, df\rangle\e_1$$
for some function $f$. Then, passing to the frame $\e_1'=\exp(-f)\e_1$, $\e_2'=\exp(f)\e_2$ we obtain $[\e_1',\e_2']=0$ from the identities (\ref{lambdamu3}).
Finally, for all $\x\in \Gamma(E^{\bot})$
\begin{align}\label{Lsection}[\x-\sqrt{-1}\Phi(\x),\e'_i]&=[\x,\e'_i]-\sqrt{-1}[\Phi(\x),\e'_i].\end{align}
Since $[\e'_1,\e'_2]=0$ and $L\oplus\C\e'_i$ is involutive, the section (\ref{Lsection}) is in $\Gamma(E^{\bot}\cap (L\oplus\C\e'_i))=\Gamma(L)$, whence
$[\Phi(\x),\e'_i]=\Phi[\x,\e'_i]$. In conclusion, the involutivity of $L$ implies the first condition in Proposition \ref{normality} and (i).
\end{proof}

\begin{prop}
An almost contact structure is normal in generalized sense if and only if it is an ordinary normal almost contact structure.
\end{prop}
\begin{proof} Let $(\Phi,\e_1,\e_2)$ be the triple associated to the almost contact structure $(\phi,\xi,\eta)$.
With respect to the decomposition ${\mathcal E}'(M)=(TM\oplus \R\partial_t)\oplus (T^*M\oplus \R dt)$, the generalized almost complex structure $J$ associated to $(\Phi,\e_1,\e_2)$ is of the form
$$J=\left[\begin{array}{cc} j &0\\0&-j^*\end{array}\right]\,,$$
where $j:TM\oplus \R\partial_t\rightarrow TM\oplus \R\partial_t$ is the almost complex structure such that
\[
j(X)=\phi(X)-\eta(X)\partial_t,\quad j(\partial_t)=\xi\,.
\]
We conclude that $(\Phi,\e_1,\e_2)$ is normal if and only if $j$ is a complex structure i.e.\ $(\phi,\xi,\eta)$ is a normal almost contact structure.
\end{proof}

\begin{prop}
An almost cosymplectic structure $(\theta,\eta)$ is normal in generalized sense if and only if $d\theta=0=d\eta$.
\end{prop}
\begin{proof}Let $(\Phi,\e_1,\e_2)$ be the triple associated to the almost contact structure $(\theta,\eta)$. Consider now the triple $(\Phi,\e_2,\e_1)$. With respect to the decomposition ${\mathcal E}'(M)=(TM\oplus \R\partial_t)\oplus (T^*M\oplus \R dt)$, the generalized almost complex structure $J$ associated to $(\Phi,\e_2,\e_1)$ is of the form
$$J=\left[\begin{array}{cc}0 &\omega\\ -\omega^{-1} & 0\end{array}\right]\,,$$
where $\omega:TM\oplus \R\partial_t\rightarrow T^*M\oplus \R dt$ is the nondegenerate 2-form such that
$$\omega(X)=i_X\theta-\eta(X)dt,\quad \omega(\partial_t)=\eta\,.$$
In conclusion, the triple $(\Phi,\e_1,\e_2)$ is normal if and only if $(\Phi,\e_2,\e_1)$ is normal, if and only if $d\omega=0$, which means
$$d\theta-d\eta\wedge dt=0\,.$$
\end{proof}

\begin{rem}In \citep{Iglesias-Wade}, a different bracket $[\ ,\ ]_1$ is introduced on  ${\mathcal E}'(M)$, which satisfies the axioms of \emph{Courant-Jacobi algebroids} \citep{Grabowsky-Marmo}. This bracket is a nontrivial extension of the Dorfman bracket on $\mathbb{T}M$ such that
$$[\alpha,X]_1=[\alpha,X]+\alpha(X)dt,\quad [\partial_t,\alpha]_1=\alpha,\quad [\partial_t,dt]_1=dt,$$
for all $X\in \Gamma(TM), \ \alpha\in \Gamma(T^*M).$
With respect to the bracket $[\ ,\ ]_1$,
\begin{itemize}
  \item an almost contact structure is normal in generalized sense if and only if it is an ordinary normal almost contact structure, and
  \item an almost cosymplectic structure $(\theta,\eta)$ is normal in generalized sense if and only if $\theta=d\eta$ (i.e., if and only if $(\theta,\eta)$ is a contact structure).
\end{itemize}
\end{rem}
It would be interesting to find a framework that includes both brackets $[\ ,\ ]_0$ and $[\ ,\ ]_1$ on ${\mathcal E}'(M)$.

\begin{prop}\label{involutivity_sphere}
Let $(E,L)$ be the generalized almost contact structure on $S^3$ associated to the functions $f,g$, defined in Example \ref{Sphere}.  Then\\
i) $(E,L)$ is integrable; \\\\
ii) $(E,L)$ is strongly integrable if and only if $L$ is involutive, if and only if
  \begin{eqnarray}\label{eqsphere1}V_2(g)+V_3(f)=0,& \quad &V_2(f)-V_3(g)=0;\end{eqnarray}\ \\
iii) the triple $(E,L)$ is normal if and only if the equations (\ref{eqsphere1}) hold and in addition
\begin{eqnarray}\label{eqsphere2}V_1(g)+2f=0,& \quad &V_1(f)-2g=0\, .\end{eqnarray}
\end{prop}
\begin{proof} With the the notation of Example \ref{Sphere}, let $\e_1=-V_1$, $\e_2=-\nu_1-fV_2-gV_3$.
Note that $E^{\bot}$ is trivialized by
$$V_2,\quad V_3,\quad \nu_3-gV_1, \quad -\nu_2+fV_1\, .$$
Moreover, $L$ is trivialized by
$$\z:=V_2-\sqrt{-1}V_3,\quad \w:=\nu_3-gV_1-\sqrt{-1}(-\nu_2+fV_1)\, .$$
Also recall that
\begin{eqnarray}\label{commutator}
[V_i,V_j]=\sum_k2\epsilon_{ij}^kV_k,&\quad& [V_i,\nu_j]=\sum_k2\epsilon_{ij}^k\nu_k.\end{eqnarray}

Then, the identities
\begin{eqnarray}
\left[\e_1, \z\right]&=&-2\sqrt{-1}\z\,,\\
\left[\e_1, \w\right]&=&-2\sqrt{-1}\w+\left(2f-2\sqrt{-1}g+V_1(g)-\sqrt{-1}V_1(f)\right)\e_1\,,\\
\left[\z, \w\right]&=&2\left(-f+\sqrt{-1}g\right)\z+\nonumber\\&&+\left((V_2(g)+V_3(f))+{\sqrt{-1}}(V_2(f)-V_3(g))\right)\e_1
\end{eqnarray}
prove (i) and that $L$ is involutive iff $[\Gamma(L),\Gamma(L)]\subset \Gamma(E^{\bot}\otimes\C)$, iff equations (\ref{eqsphere1}) are satisfied. Moreover, from
\begin{eqnarray}
\left[\e_2, \z\right]&=&-2\w+\left(V_2(f)-\sqrt{-1}V_3(f)\right)\z+\nonumber\\ &&+\left(V_2(g)+V_3(f)+\sqrt{-1}V_2(f)-\sqrt{-1}V_3(g)\right) V_3 \\
\left[\e_2, \w \right]&=&\left(V_1(g)+2f+\sqrt{-1}(V_1(f)-2g)\right)\e_2-\left(V_3(g)+\sqrt{-1}V_3(f)\right)\w+\nonumber\\
       &&+\left(-gV_1(f)+2g^2+fV_1(g)+2f^2\right)\z+\nonumber\\
       &&+\left(-fV_2(g)-fV_3(f)-\sqrt{-1}fV_2(f)+\sqrt{-1}fV_3(g)\right)\e_1+\nonumber\\
       \label{eqncoeffs}&&-\left(V_2(g)+V_3(f)+\sqrt{-1}(V_2(f)-V_3(g))\right)\nu_2
\end{eqnarray}
we conclude that  $L\oplus\C\e_2$ is involutive if and only if equations (\ref{eqsphere1}) are satisfied, which concludes the proof of (ii). To prove (iii), notice that
\[
[\e_1,\e_2]=[V_1, \nu_1+fV_2+gV_3]=V_1(f)V_2+V_1(g)V_3+2fV_3-2gV_2\,,
\]
which is the projection onto $E^{\bot}$ of an exact 1-form if and only if it is equal to $0$, if and only if equations (\ref{eqsphere2}) hold.
\end{proof}

\begin{examples}\ \\\begin{itemize}
\item[(a)] Let $h(z,w)$ be a holomorphic function in a neighborhood of $S^3\subset\C^2$, where $\C^2$ is identified with $\R^4$ with complex coordinates $z=x_1+\sqrt{-1}x_2$, $w=x_3+\sqrt{-1}x_4$. Then, $f={\rm Re}(h)$, $g={\rm Im}(h)$ satisfy the equations (\ref{eqsphere1}) and define a strongly integrable generalized almost contact structure $(E,L)$ on $S^3$, whose type jumps along $$S^3\cap\{h(z,w)=0\}.$$
Moreover, $(E,L)$ is normal if and only if $$\partial_t(g)=2g,\quad \partial_t(f)=2f.$$
\item[(b)] In particular if $h(z,w)$ is a homogeneous polynomial then the associated generalized almost contact structure is strongly integrable, its type jumps along a link and it is normal if and only if either $h\equiv 0$ of ${\rm deg}(h)=2$.
\end{itemize}
\end{examples}
\begin{rem}
Let $H$ be a closed three-form on $M$. As in section \ref{Cone}, the twisted Dorfman bracket $[\ ,\ ]_H$ induces a bracket on ${\mathcal E}'(M)$and corresponding notions of $H$-\emph{normal triples} and (\emph{strongly}) $H$-\emph{integrable} generalized almost contact structures. It is easy to see that Proposition \ref{normality} and Corollary \ref{normalinv} extend to the twisted setting.\end{rem}
\begin{example}Consider the generalized almost contact structure on $S^3$ as in Example \ref{Sphere} and set $H=c\nu_1\wedge\nu_2\wedge\nu_3$. Then
\begin{itemize}
  \item $(E,L)$ is always $H$-integrable;
  \item $(E,L)$ is strongly $H$-integrable if and only if
$$V_2(g)+V_3(f)=c(g^2-f^2),\quad V_2(f)-V_3(g)=2cfg;$$
  \item $(E,L)$ is $H$-normal if and only if  $c=0$ or $f=0=g$.
\end{itemize}
\end{example}

\section{Mixed pairs}

Recall that given a manifold $M$ of real dimension $m$, the Clifford algebra of $(\T M\otimes \mathbb C,\langle\,,\,\rangle)$ acts through its standard representations on differential forms $\Omega^\bullet(M) = \bigwedge^\bullet T^* M\otimes \C$. To avoid cluttering the notation, we simply write $\rho_1\rho_2$ for the wedge product $\rho_1\wedge\rho_2$ in the DGA $(\Omega^\bullet(M),d)$. A differential form $\rho\in \Omega^\bullet(M)$  is a {\it pure spinor} if its annihilator ${\rm Ann}(\rho)\subseteq \T M$ is a maximal isotropic subbundle. Let ${\rm pr}^{m}$ be the projection of $\Omega^\bullet(M)$ onto $\Omega^m(M)$. The $\Omega^{{\rm top}}(M)$-valued bilinear pairing $\mu_M$ on $\Omega^\bullet(M)$ defined by $$\mu_M(\rho_1,\rho_2)=(-1)^{{m\choose 2}} {\rm pr}^m(\rho_1\rho_2)$$ for all $\rho_1,\rho_2\in \Omega(M)$.  The {\it type} of a pure spinor $\rho\in \Omega^\bullet (M)$ is the function ${\rm type}(\rho)$ whose value at $x\in M$ is the minimum of the degrees of the components of $\rho(x)\in T^*_x M$. We refer the reader to \citep{Gualtieri} for a systematic treatment of pure spinors and their applications to generalized complex geometry.

Assume now that $m=2n+1$.
Let $(E,L)$ be a generalized almost contact structure represented by the triple $(\Phi,\e_1,\e_2)$. Then $L\oplus\C\e_1$ and $L\oplus\C \e_2$ are maximally isotropic subbundles of $\T M\otimes \C$. Therefore, at any point $x\in M$ there exist $\rho_1,\rho_2\in \Omega^\bullet(M)$ such that
\[
L_x\oplus\C\e_1={\rm Ann}(\rho_1)\,;\quad L_x\oplus\C\e_2={\rm Ann}(\rho_2)\,.
\]
In particular, this implies $\mu_M(\rho_1,\bar \rho_2)\neq 0$.
\begin{mydef}
  A pair of pure spinors $(\rho_1,\rho_2)\in \Omega^\bullet(M)\times \Omega^\bullet(M)$ is called {\it non-degenerate} if $\mu_M(\rho_1,\bar \rho_2)\neq 0$. We call a non-degenerate pair $(\rho_1,\rho_2)$ of pure spinors a {\it mixed pair} if there exist isotropic sections $\e_1,\e_2$ of $\T M$ such that $\rho_1 = \e_1 \cdot \rho_2$ and $\rho_2=\e_2\cdot \rho_1$.
\end{mydef}
\begin{rem}
Let  $(\rho_1,\rho_2)$ be a mixed pair and and let $\e_1,\e_2$ be isotropic sections as above. Then, the bundle $E$ generated by $\e_1,\e_2$ does not depend on the choice of such sections. Moreover, if we set
\[
L={\rm Ann}(\rho_1)\cap {\rm Ann}(\rho_2)\,,
\]
then $(E,L)$ is a generalized almost contact structure.\end{rem}

\begin{mydef}Let $(M,H)$ be a pair consisting of a manifold $M$ and a real closed 3-form $H$ on $M$, and let $d_H$ denote the twisted de Rham differential $d + H\cdot$. A mixed pair $(\rho_1,\rho_2)$ is {\it integrable} if there exists $i\in\{1,2\}$ and $\bf v\in\Gamma(\T M)$ such that
\[
d_H \rho_i = {\bf v} \cdot \rho_i
\]
and {\it strongly integrable} if there exist ${\bf v}_1, {\bf v}_2\in\Gamma(\T M)$ such that  \[
d_H \rho_1 = {\bf v}_1 \cdot \rho_1, \quad d_H \rho_2 = {\bf v}_2 \cdot \rho_2.
\]
\end{mydef}

\begin{rem}
As pointed out to us by Tomasiello,  our definition of strongly integrable mixed pairs is somewhat reminiscent of the supersymmetry equations for type II supergravity solutions on ${\rm AdS}_7\times M_3$ \citep{AFRT}. It would be interesting to make this connection more precise.
\end{rem}

\begin{example}
If $(\theta,\eta)$ is an almost cosymplectic structure on $M$ and  $\rho_1=e^{\sqrt{-1}\theta}$ and $\rho_2=\rho_1\eta$, then $(\rho_1,\rho_2)$ is a mixed pair. Moreover, $(\rho_1,\rho_2)$ is integrable if $d\theta =0$ and strongly integrable if $d\theta=0=d\eta$.
\end{example}

\begin{lem}\label{Lemma}
Let $\rho_1,\rho_2\in \Omega^\bullet(M)$ be a pair of pure spinors of definite and opposite parity and let $\rho = \rho_1 + \sqrt{-1} dt\rho_2 \in \Omega^\bullet (C(M))$. Then $\mu_{C(M)}(\rho,\overline{\rho})\neq 0$ if and only if $(\rho_1,\rho_2)$ is non-degenerate.

\end{lem}
\begin{proof}  $\partial_t\cdot \rho_1=0$ implies  $\mu_{C(M)}(\rho_1,\bar \rho_1)=0$. If $\epsilon = e^{\pi \sqrt{-1} {m\choose 2}}$, then using the bilinearity and symmetry of Mukai pairings,
\begin{eqnarray*}
\mu_{C(M)}(\rho,\bar \rho) &=& \mu_{C(M)}(\rho_1,-\sqrt{-1} dt\rho_2)+ \mu_{C(M)} (\sqrt{ -1}dt\rho_2,\bar \rho_1)\\
&=&\sqrt{ -1}(-1)^{|\rho_2|}dt(\mu_M(\rho_1,\bar \rho_2)+\epsilon\mu_M(\bar \rho_1,\rho_2)),\\
\end{eqnarray*}
from which the result follows.
\end{proof}

\begin{theorem}\label{spinortheorem}
Let $\rho_1,\rho_2\in \Omega^\bullet (M)$, and let $\rho=\rho_1+\sqrt{-1}dt\rho_2\in \Omega^\bullet (C(M))$. Then
\begin{itemize}
  \item[i)] $(\rho_1,\rho_2)$ is a mixed pair if and only if $\rho$ is a pure spinor defining a generalized almost complex structure on $C(M)$;
  \item[ii)] The mixed pair $(\rho_1,\rho_2)$ is integrable if and only if there exists ${\bf v}\in \Gamma(\T M\otimes \C)$ such that
\[
\partial_t \cdot dt\cdot d\rho = {\bf v}\cdot \rho\,.
\]
  \item[iii)] ${\rm Ann}(\rho)$ is involutive if and only if there exists ${\bf w}\in \Gamma(\T M\otimes \C)$ such that
\begin{equation}\label{normalspinor}d\rho_1={\bf w}\cdot\rho_1,\quad  d\rho_2={\bf w}\cdot\rho_2.\end{equation}
\end{itemize}
\end{theorem}
\begin{proof} Proof of (i): Suppose $(\rho_1,\rho_2)$ is such a mixed pair. Then $(\C\e_2)^\perp\cap {\rm Ann} (\rho_1)\subset {\rm Ann} (\rho)$. Since $\rho_1$ is pure, then $\dim {\rm Ann}(\rho)\ge m -1$. On the other hand, $dt-\sqrt{ -1}\e_1$ and $\partial_t-\sqrt{ -1}\e_2$  are linearly independent sections of ${\rm Ann}(\rho)\setminus {\rm Ann}(\rho_1)$ and thus $\rho$ is pure. Since $\rho_1= \e_1\cdot \rho_2$ and $\rho_2$ have opposite parity,  Lemma \ref{Lemma} shows that $\rho$ defines a generalized almost complex structure $J$ on $C(M)$. Conversely, if $\rho$ is pure and $\mu_{C(M)}(\rho,\overline{\rho})\neq 0$, then by Lemma \ref{Lemma} $(\rho_1,\bar \rho_2)$ is non-degenerate. Moreover, $\rho_2=-\sqrt{ -1}\partial_t\cdot \rho$ implies that
\begin{equation}\label{An}
(\C \partial_t)^\perp \cap {\rm Ann}(\rho) \subseteq {\rm Ann}(\rho_2)\,.
\end{equation}
Therefore, $\rho_2$ is pure. Similarly, $\rho_1 = \partial_t \cdot dt \cdot \rho$ implies
\begin{equation}\label{Ann}
({\rm span}(\partial_t, dt))^\perp \cap {\rm Ann}(\rho) \subseteq {\rm Ann}(\rho_1)
\end{equation}
and therefore $\dim {\rm Ann}(\rho_1) \ge m -1$. Since $dt-\sqrt{-1}J(dt)\in {\rm Ann}(\rho)$, one has $dt\cdot \rho_1 = \sqrt{-1} J(dt)\cdot \rho$. Since $dt$ is isotropic and $J$ is orthogonal,
\[
J(dt)\cdot dt\cdot \rho_1 =0\, .
\]
Using
\[
\langle J(dt),dt\rangle = \langle -dt,J(dt)\rangle =0\,,
\]
we conclude that $J(dt)\in {\rm Ann}(\rho_1)$. Similarly, $\partial_t-\sqrt{-1}J(\partial_t)\in {\rm Ann}(\rho)$ implies $J(\partial_t)\in {\rm Ann}(\rho_2)$. In particular,
\[
0=J(dt)\cdot J(\partial_t)\cdot \rho_2= \rho_2-J(\partial_t)\cdot J(dt)\cdot \rho_2
\]
and since $\rho_2\neq 0$, we deduce that $J(dt)\in{\rm Ann}(\rho_1)\setminus {\rm Ann}(\rho_2)$. This proves that $\rho_1$ is pure for otherwise $\dim {\rm Ann}(\rho_1)= m -1$ which together with (\ref{An}) and (\ref{Ann}) would imply ${\rm Ann}(\rho_1)\subseteq {\rm Ann}(\rho_2)$, contradicting $J(dt)\neq 0$. From
$$0=\partial_t\cdot\left(dt-\sqrt{-1}J(dt)\right)\cdot\rho=\rho_1-J(dt)\cdot\rho_2$$
it follows that $0=\langle\partial_t,J(dt)\rangle=-\langle dt,J(\partial_t)\rangle$ and thus
$$0=dt\cdot\left(dt-\sqrt{-1}J(dt)\right)\cdot\rho=\sqrt{-1}dt\cdot\left(\rho_2-J(dt)\cdot\rho_1\right).$$
The result follows upon setting $\e_1= J(dt)$ and $\e_2 =  J(\partial_t)$. \\
Proof of (ii): it follows from
\[
\partial_t\cdot dt\cdot d\rho = d\rho_1
\]
and the definition of integrability of mixed pairs.\\
Proof of (iii): if equations (\ref{normalspinor}) hold, then $$d\rho=d\rho_1-\sqrt{-1}dtd\rho_2=\w\cdot\rho_2-\sqrt{-1}dt\w\cdot\rho_2=\w\cdot\left(\rho_1+\sqrt{-1}dt\rho_2\right),$$
whence ${\rm Ann}(\rho)$ is involutive. Conversely, suppose there exist $\mathbf{v}\in\Gamma(\T M\otimes \C)$, $f,g\in C^{\infty}(M)$ such that
\begin{align}d\rho_1-\sqrt{-1}dtd\rho_2&=\left(\mathbf{v}+f\partial_t+gdt\right)\cdot (\rho_1+\sqrt{-1}dt\rho_2)\nonumber\\
&=(\mathbf{v}+\sqrt{-1}f\e_2)\cdot\rho_1+dt\left(-\sqrt{-1}\mathbf{v}+g\e_1\right)\cdot\rho_2.\nonumber\end{align}
Then, equations (\ref{normalspinor}) hold, with $\w=\mathbf{v}+\sqrt{-1}\left(f\e_2+g\e_1\right)$.

\end{proof}
\begin{example}\label{spinorsphere}
In the Example \ref{Sphere}, it is easy to check that one can choose
\[
\rho_1 = \sqrt{-1}\nu_2+\nu_3;\qquad \rho_2=(g+\sqrt{-1}f)+\nu_1(\sqrt{-1})\nu_2+\nu_3)\,.
\]
From equation (\ref{commutator}), $d\nu_k = 2\epsilon_{ij}^k \nu_i\nu_j$ which implies
\[
d\rho_1 = -2\sqrt{-1}\nu_1\rho_1 = (2 \sqrt{-1} \e_2 + (g+\sqrt{-1}f)(V_2+\sqrt{-1}V_3)) \cdot \rho_1\,.
\]
Therefore, the mixed pair $(\rho_1,\rho_2)$ is integrable, as expected from Proposition \ref{involutivity_sphere}\,.

\end{example}

\begin{prop}\label{types}
Let $(\rho_1,\rho_2)$ be a mixed pair, and let $(E,L)$ be the corresponding generalized almost contact structure. Then
\[
2t_L ={\rm type}(\rho_1)+{\rm type}(\rho_2)+1.
\]

\end{prop}
\begin{proof} For $i=1,2$,
\begin{equation}\label{t_i}
{\rm type}(\rho_i) = \dim M -\dim a(\C \e_i) -\dim a(L) + \dim (a(L)\cap a(\C \e_i))
\end{equation}
implies
\begin{equation*}
{\rm type}(\rho_i)-t_L =
\left\{
\begin{array}{lll}
-1 & {\rm if} & a(\e_i)\notin a(L)\\
&&\\
0 &{\rm if}& a(\e_i)\in a(L)
\end{array}
\right.
\end{equation*}
and therefore
\begin{equation*}
|{\rm type}(\rho_1)-{\rm type}(\rho_2)| =
\left\{
\begin{array}{lll}
1 & {\rm if} & \delta =1\\
&&\\
0 &{\rm if}& \delta =0
\end{array}
\right.
\end{equation*}
where
\[
\delta = p_E - \dim (a(E)\cap a(L))\,.
\]
On the other hand, $\rho_1 = \e_1\cdot \rho_2$ implies that $\delta$ always equals one. Summing the equations in (\ref{t_i}),
\[
{\rm type}(\rho_1)+{\rm type}(\rho_2) = 2 t_L - \delta = 2t_L -1\,.
\]
\end{proof}


\section{T-Duality}
We begin by recalling some definitions and notations from \citep{Gualtieri-Cavalcanti} (see also \citep{BEM1}, \citep{BEM2}). Let $\pi:M\to B$ and $\tilde\pi:\tilde M\to B$ be principal bundles with $k$-dimensional torus fibers $T^k$ and let  $p$ and $\tilde p$ denote the projections of the fiber product $M\times _B \tilde M$ onto $M$ and $\tilde M$, respectively. We denote by $\Omega_{T^k}^\bullet(M)$ and $\Omega_{T^k}^\bullet(\tilde M)$ the $T^k$-invariant differential forms on $M$ and $\tilde M$, respectively. Given $H\in \Omega_{T^k}^3(M)$ and $\tilde H\in\Omega_{T^k}^3(\tilde M)$, we say that $(M,H)$ and $(\tilde M, \tilde H)$ form a {\it T-dual pair} with $k$-dimensional fibers and {\it pairing} $F\in \Omega_{T^{2k}}^2 (M\times_B\tilde M)$ if $dF=p^*H-\tilde p^* \tilde H$ and $F$ is non-degenerate as a pairing between vertical tangent vectors of $M$ and vertical tangent vectors of $\tilde M$. Given a T-dual pair $(M,H)$, $(\tilde M,\tilde H)$  with $k$-dimensional fibers and with pairing $F$, define
\[
\tau_F:(\Omega_{T^k}^\bullet (M),d_H)\to (\Omega_{T^k}^\bullet(\tilde M),d_{\tilde H})
\]
by the formula
\[
\tau_F(\rho)=\int_{T^k} e^F  p^*(\rho)\,.
\]
Consider also the map

\[
\varphi_F: \T M/{T^k}\to \T\tilde M/{T^k}
\]
defined by $\varphi_F(X+\xi) = \tilde p_*\left(\hat X + p^*\xi - F(\hat X)\right)$, where $\hat X$ is the unique lifting of $X$ to $M\times_B\tilde M$ such that $p^*\xi - F(\hat X)$ is basic with respect to $\tilde p$.
\begin{theorem}[\cite{BEM1},\cite{BEM2},\cite{Gualtieri-Cavalcanti}]\label{Tduality} Let $(M,H)$ and $(\tilde M,\tilde H)$ be a T-dual pair with $k$-dimensional fibers and pairing $F$. Let  also $\tau_F$ and $\varphi_F$ be as above. Then
\begin{itemize}
\item[i)] $\tau_F$ is an isomorphism of complexes;
\item[ii)] $\varphi_F$ is an isomorphism of Courant algebroids;
\item[iii)] for all ${\bf v}\in \Gamma(\T M)$ and $\rho\in \Omega_{T^k}^\bullet(\tilde M)$,
\[
\tau_F({\bf v}\cdot \rho)=\varphi_F({\bf v})\cdot \tau_F(\rho)\,.
\]
\item[iv)] If $\rho=e^{\sqrt{-1} \omega} \Omega\in \Omega_{T^k}^\bullet (M)$ is a pure spinor and  $j$ is the smallest integer such that
\[
\int_{T^k} (F+\sqrt{-1}\omega)^j \Omega \neq 0\,,
\]
then
\[
{\rm type}(\tau_F(\rho))={\rm type}(\rho) + 2j-k\,.
\]
In particular when $k=1$,  ${\rm type}(\tau_F(\rho))-{\rm type}(\rho)$ equals $1$ if $\Omega$ is basic and $-1$ otherwise.
\end{itemize}
\end{theorem}

\begin{theorem}\label{Tdualcontact}
Let $(M,H)$ and $(\tilde M,\tilde H)$ be a T-dual pair with $k$-dimensional fibers and pairing $F$.
\begin{itemize}
  \item[i)] If $(E,L)$ is a $T^k$-invariant generalized almost contact structure on $M$, then $(\varphi_F(E),\varphi_F(L))$ is a $T^k$-invariant generalized almost contact structure on $\tilde M$.
\item[ii)] If $(\Phi,\e_1,\e_2)$ is a $T^k$-invariant triple on $M$, then $(\varphi_F\circ\Phi\circ \varphi_F^{-1},\varphi_F(\e_1),\varphi_F(\e_2))$ is a $T^k$-invariant triple on $\tilde M$.

\item[iii)] The map $\varphi_F$ preserves twisted (strong) integrability and twisted normality.
\item[iv)] If $(\rho_1,\rho_2)$ is a $T^k$-invariant mixed pair on $(M,H)$, then $(\tau_F (\rho_1),\tau_F(\rho_2))$ is a $T^k$-invariant mixed pair on $(\tilde M,\tilde H)$.
\item[v)] Let $\rho_1=e^{\sqrt{-1}\omega_1}\Omega_1,\ \rho_2=e^{\sqrt{-1}\omega_2}\Omega_2\in \Omega_{T^k}^\bullet (M) $ be such that $(\rho_1,\rho_2)$ is a mixed pair and let $(E,L)$ be the corresponding generalized almost contact structure. If $j_1, j_2$ are the smallest integers such that
\[
\int_{T^k} (F+\sqrt{-1} \omega_i)^{j_i} \Omega_i \neq 0\,,
\]
for $i=1,2$, then
\[
t_{\varphi_F(L)}-t_L = j_1 + j_2 - k\,.
\]
In particular, if $k=1$, then $t_{\varphi_F(L)}-t_L$ equals $1$ if $\Omega_1$ and $\Omega_2$ are basic, $-1$ if $\Omega_1$ and $\Omega_2$ are both not basic and $0$ otherwise. Moreover,
\[
|p_{\varphi_F(E)}-p_E| =
\left\{
\begin{array}{rl}
0& {\rm if}\quad  \pi_*(a(E))=0=\pi_*(a(\varphi_F(E))) \\
& \\
1 & {\rm otherwise}\,.
\end{array}
\right.
\]
\end{itemize}
\end{theorem}
\begin{proof}The first three statements are a direct consequence of the previous theorem. We now prove (iv): the torus fibration on $M$ and $\tilde M$ extend trivially to $C(M)$ and $C(\tilde M)$ in such a way that $C(M)$ and $C(\tilde M)$ are T-dual with respect to the obvious duality pairing defined by $F$ (which we denote by the same letter). Moreover, the pure spinor $\rho=\rho_1+\sqrt{ -1}dt\rho_2$ on $C(M)$ is $T^k\times \R^+$-invariant and Theorem \ref{Tduality} implies that
\[
\tau_F(\rho)=\tau_F(\rho_1)+\sqrt{ -1}dt\tau_F(\rho_2)
\]
is a $T^k\times \R^+$-invariant pure spinor. From Theorem \ref{spinortheorem} we conclude that $(\tau_F(\rho_1),\tau_F(\rho_2))$ is a $T^k$-invariant mixed pair.\\
Finally, we prove (v): the second statement follows directly from the definition of $p_E$ and of $\varphi_F$. For the first, let $t_i={\rm type}(\rho_i)$ and $\hat t_i={\rm type}(\tau_F(\rho_i))$ for $i=1,2$. Then, Theorem \ref{Tduality} and Proposition \ref{types} imply
\[
2(t_{\varphi_F(L)}-t_L)= \hat t_1+\hat t_2 - t_1 - t_2 = 2j_1 +2j_2 - 2k\,,
\]
which concludes the proof.
\end{proof}

\begin{prop}
Let $(\rho_1,\rho_2)$ be a mixed pair on $M$. Let $F=-dtd\tilde t$ be the standard pairing on $S^1\times S^1$ and let $\tau_F$ denote the induced T-duality automorphism acting on $\Omega^\bullet_{S^1}(M\times S^1)$.  Then $(\tau_F(\rho_1),\tau_F(\rho_2))$ = $(\rho_2,\rho_1)$.
\end{prop}
\begin{proof} Notice that $\rho=\rho_1+\sqrt{ -1}dt\rho_2$, originally defined as a spinor on $C(M)$, also defines a spinor on $M\times S^1$. Then
\[
\tau_F(\rho)=\int_{S^1} e^F \rho = \sqrt{-1}(-1)^{|\rho_2|}(\rho_2+\sqrt{-1}dt\rho_1).
\]
\end{proof}

\begin{prop}
Let $(E,L)$ be a generalized almost contact structure on $M$. If, for some $\omega\in \Omega^2(M\times S^1)$ and $\Omega\in \Omega^{t_J}(M\times S^1)$, $\rho=e^{\sqrt{-1}\omega}\Omega$ is the corresponding pure spinor on $M\times S^1$, then $0\le t_L-t_J\le 1$. Moreover, the following are equivalent
\begin{itemize}
\item[i)] $t_L=t_J$,
\item[ii)] $\int_{S^1} \Omega \neq 0\,$ ,
\item[iii)] ${\rm type}(\rho_1)-{\rm type}(\rho_2) = 1$\,.
\end{itemize}
\end{prop}
\begin{proof} Let $F=-dt d\tilde t$ be the standard duality pairing on $M\times S^1\times S^1$ and let $\hat t_J = {\rm type}(\tau_F(\rho))$. According to Theorem \ref{Tduality}
\[
\hat t_J -t_J=
\left\{
\begin{array}{rll}\label{that}
-1 & {\rm if} &  \int_{S^1} \Omega \neq 0 \\
& &\\
1 & {\rm if} & \int_{S^1} \Omega = 0\,.
\end{array}
\right.
\]
On the other hand, $\rho=\rho_1 + \sqrt{-1} dt\cdot \e_2\cdot \rho_1$ implies that $t_J={\rm type}(\rho_1)$ and $\hat t_J= {\rm type}(\rho_2)$. Combining (\ref{that}) and  Proposition \ref{types},
\[
t_L-t_J =
\left\{
\begin{array}{rll}
0& {\rm if} &  \int_{S^1} \Omega \neq 0 \\
& &\\
1 & {\rm if} & \int_{S^1} \Omega = 0\,.
\end{array}
\right.
\]
\end{proof}

\begin{examples}\ \\
\begin{itemize}
\item[(a)](Three dimensional Heisenberg group)
Let $H_3$ be the real three-dimensional Heisenberg group with
Lie algebra
\[
\mathfrak{h}_3 = {\rm Lie}(H_3)= {\rm span}(X_1,X_2,X_3)
\]
has $[X_1,X_2]=-X_3$ as the only nontrivial bracket. If $\mathfrak{h}_3^*={\rm span}(\alpha^1,\alpha^2,\alpha^3)$ and $d\alpha^3 = \alpha^1\alpha^2$. It follows that the remaining nontrivial Dorfman brackets are $[X_1,\alpha^3] = \alpha^2$ and $[X^2,\alpha^3]=-\alpha^1$. Let $M$ be the left quotient of $H_3$ by the discrete subgroup ${\rm exp}(\mathbb Z X_3)$. Then $M$ is T-dual to $(\tilde M,\tilde H)$ where $\tilde M=S^1\times \mathbb R^2$ and $\tilde H=\alpha^1\alpha^2\tilde\alpha^3$, written in terms of a basis  $\{\alpha^1,\alpha^2,\tilde \alpha^3\}$ of global sections for $T^*\tilde M$ dual to $\{X_1,X_2,\tilde X_3\}$. The corresponding nontrivial brackets on $\T\tilde M$ are $[X_1,X_2]_{\tilde H} = \tilde \alpha^3$, $[X_1,\tilde X_3]_{\tilde H} = -\alpha^2$ and $[X_2,\tilde X_3]_{\tilde H} = \alpha^1$.  Therefore, $\varphi_F:\T M\to \T \tilde M$ is defined on generators as $\varphi_F(\alpha^3)= -\tilde X_3$, $\varphi_F(X_3)= -\tilde \alpha^3$ and for $i=1,2$ $\varphi_F(X_i)=X_i$, $\varphi_F(\alpha^i)=\alpha^i$. Consider the family of cosymplectic structure on $M$ parametrized by $b,c\in \mathbb R$ with $E_{b,c} = {\rm span}(X_1-cX_2+ bX_3, \alpha^1)$ and
\[
L_{b,c}={\rm span}(X_2-\sqrt{-1}\alpha^3+\sqrt{-1}b\alpha^1,X_3+\sqrt{-1}\alpha^2+\sqrt{-1}c\alpha^1)\, .
\]
Then $\tilde E_{b,c}=\varphi_F(E_{b,c})={\rm span}(X_1-cX_2- b\tilde\alpha^3, \alpha^1)$ and
\[
\tilde L_{b,c}=\varphi_F(L_{b,c})={\rm span}(X_2+\sqrt{-1}\tilde X_3+\sqrt{-1}b\alpha^1,-\tilde\alpha^3+\sqrt{-1}\alpha^2+\sqrt{-1}c\alpha^1)\, .
\]
Notice that the geometric type of $(E_{b,c},L_{b,c})$ is $(1,1)$ while the geometric type of $(\tilde E_{b,c},\tilde L_{b,c})$ is $(1,2)$. Moreover for $b\neq 0$, $(\tilde E_{b,c},\tilde L_{b,c})$ is not represented by a Poon-Wade triple.

\item[(b)](T-dual of $S^3$)
Consider the generalized almost complex structure defined in Example \ref{Sphere} and assume that $V_1(f)=V_1(g)=0$. Since the $S^1$-action generated by $V_1$ has no fixed points on $S^3$ the quotient map $\pi:S^3\to S^2=S^3/S^1$ is a (Hopf) fibration.  With respect to the projection $\tilde \pi: S^1\times S^2\to S^2$, $(S^3,0)$ and $(S^1\times S^2, \tilde \nu_1\nu_2\nu_3)$ define a T-dual pair with 1-dimensional fibers and pairing $F=-\nu_1\tilde \nu_1$. From Example \ref{spinorsphere}
\[
\tau_F(\rho_1) = -\tilde \nu_1 (\sqrt{-1} \nu_2 + \nu_3)\,,\quad \tau_F(\rho_2)=-(g+\sqrt{-1} f) \tilde \nu_1 + \sqrt{-1}\nu_2 + \nu_3\,,
\]
which implies $t_{\varphi_F(L)}=2$. On the other hand,
\[
\varphi_F(\e_1) = -\tilde \nu_1\,,\quad \varphi_F(\e_2)=-\tilde V_1-fV_2-gV_3\,,
\]
where $\tilde V_1=\varphi_F(\nu_1)$, which implies $p_{\varphi_F(E)}=1$. Notice that while $(p_E,t_L)$ may jump, $(p_{\varphi_F(E)},t_{\varphi_F(L)})$ is constant. Moreover,
\[
t_{\varphi_F(L)}-t_L = |p_{\varphi_F(E)}-p_E|=\left\{\begin{array}{cl}0&\mbox{ if } f(x)=0=g(x),\\\\
                                    1&\mbox{ otherwise,}\end{array}\right.
\]
which shows (upon considering the inverse T-duality as well) that all the cases listed in Theorem \ref{Tdualcontact} for $k=1$ can occur.\end{itemize}\end{examples}
\section{Acknowledgements}{We are indebted to Alessandro Tomasiello for valuable comments on the first draft. M.A.\ would like to thank the organizers of the ``Geometry of Strings and Fields" program at the Galileo Galilei Institute for Theoretical Physics, where parts of this paper were written, for providing excellent working conditions. D.G. is grateful to the Department of Mathematics and Statistics of the University of New Mexico and to the Efroymson Family Foundation for their support.}

\bibliography{GcontactBib}{}

\begin{thebibliography}{10}

\bibitem{AFRT}
Fabio Apruzzi, Marco Fazzi, Dario Rosa, and Alessandro Tomasiello.
\newblock All {$AdS_7$} solutions of type {II} supergravity, 2013.
\newblock arXiv:1309.2949.

\bibitem{BEM1}
Peter Bouwknegt, Jarah Evslin, and Varghese Mathai.
\newblock {$T$}-duality: topology change from {$H$}-flux.
\newblock {\em Comm. Math. Phys.}, 249(2):383--415, 2004.

\bibitem{BEM2}
Peter Bouwknegt, Jarah Evslin, and Varghese Mathai.
\newblock Topology and {$H$}-flux of {$T$}-dual manifolds.
\newblock {\em Phys. Rev. Lett.}, 92:181601, May 2004.

\bibitem{Gualtieri-Cavalcanti}
Gil~R. Cavalcanti and Marco Gualtieri.
\newblock Generalized complex geometry and {$T$}-duality.
\newblock In {\em A celebration of the mathematical legacy of {R}aoul {B}ott},
  volume~50 of {\em CRM Proc. Lecture Notes}, pages 341--365. Amer. Math. Soc.,
  Providence, RI, 2010.

\bibitem{GT}
Ralph~R. Gomez and Janet Talvecchia.
\newblock On products of generalized geometries, 2014.
\newblock arXiv:1311.0381.

\bibitem{Grabowsky-Marmo}
Janusz Grabowski and Giuseppe Marmo.
\newblock The graded {J}acobi algebras and (co)homology.
\newblock {\em J. Phys. A}, 36(1):161--181, 2003.

\bibitem{Gualtieri}
Marco Gualtieri.
\newblock Generalized complex geometry.
\newblock {\em Ann. of Math. (2)}, 174(1):75--123, 2011.

\bibitem{Iglesias-Wade}
David Iglesias-Ponte and A{\"{\i}}ssa Wade.
\newblock Contact manifolds and generalized complex structures.
\newblock {\em J. Geom. Phys.}, 53(3):249--258, 2005.

\bibitem{Poon-Wade}
Yat~Sun Poon and A{\"{\i}}ssa Wade.
\newblock Generalized contact structures.
\newblock {\em J. Lond. Math. Soc. (2)}, 83(2):333--352, 2011.

\bibitem{Vaisman}
Izu Vaisman.
\newblock Generalized crf-structures.
\newblock {\em Geometriae Dedicata}, 133(1):129--154, 2008.

\end{thebibliography}
\bibliographystyle{plain}

\end{document}